\documentclass[12pt]{amsart}
\usepackage[all]{xy}
\usepackage{amssymb}
\usepackage{amsthm}
\usepackage{amsmath}
\usepackage{amscd,enumitem}
\usepackage{verbatim}
\usepackage{eurosym}
\usepackage{graphicx}
\usepackage{dsfont}
\usepackage{stmaryrd}
\usepackage{color}
\usepackage{float}
\usepackage{longtable}
\usepackage{dcolumn}
\usepackage[mathscr]{eucal}
\usepackage[all]{xy}
\usepackage{hyperref}
\usepackage[title]{appendix}
\usepackage[usenames,dvipsnames]{xcolor}
\usepackage{bbm}
\usepackage[textheight=8.85in, textwidth=6.8in]{geometry}
\usepackage{multirow}
\usepackage{caption}
\theoremstyle{plain}
\newtheorem{theorem}{Theorem}[section]
\newtheorem{corollary}[theorem]{Corollary}
\newtheorem{lemma}[theorem]{Lemma}
\newtheorem{proposition}[theorem]{Proposition}

\theoremstyle{definition}

\theoremstyle{remark}
\newtheorem*{remark}{Remark}

\newcommand{\Z}{\mathbb{Z}}
\newcommand{\Q}{\mathbb{Q}}

\newcommand{\F}{\mathbb{F}}

\newcommand{\im}{\mathrm{Im}}
\newcommand{\re}{\mathrm{Re}}
\newcommand{\hes}{\text{{\rm{Hes}}}}

\newcommand{\I}{\mathcal{I}_\lambda^{\text{\hes}}}

\newcommand{\R}{\mathbb{R}}
\newcommand{\HH}{\mathbb{H}}

\newcommand{\N}{\mathbb{N}}
\newcommand{\SL}{\operatorname{SL}}

\newcommand{\C}{\mathbb{C}}

\newcommand{\leg}[2]{\genfrac{(}{)}{}{}{#1}{#2}}
\newcommand{\s}{\operatorname{sc}_7}

\catcode`,\active

\catcode`\,12

\numberwithin{equation}{section}

\begin{document}
	\title[Distribution of the Hessian values of Gaussian hypergeometric functions]{Distribution of the Hessian values of Gaussian hypergeometric functions}
	\dedicatory{Dedicated to George Andrews and Bruce Berndt for their 85th birthdays}
	\author{Ken Ono, Sudhir Pujahari, Hasan Saad \and Neelam Saikia}
	
	\address{Department of Mathematics, University of Virginia, Charlottesville, VA 22904}
	\email{ko5wk@virginia.edu}
	
	\address{School of Mathematical Sciences, National Institute of Science Education and Research, Bhubaneswar, An OCC of Homi Bhabha National Institute,  P. O. Jatni,  Khurda 752050, Odisha, India}
	\email{spujahari@niser.ac.in}
	
	\address{Department of Mathematics
		Louisiana State University
		Baton Rouge, LA 70803}
	\email{hsaad@lsu.edu}
	
	\address{School of Basic Sciences, Indian Institute of Technology Bhubaneswar, Argul, Khordha 752050, Odisha, India}
	\email{neelamsaikia@iitbbs.ac.in}

	
	\begin{abstract}
		We consider a special family of Gaussian hypergeometric functions whose entries are cubic and trivial characters over finite fields. The special values of these functions are known to give the Frobenius traces of families of Hessian elliptic curves.  Using the theory of harmonic Maass forms and mock modular forms, we prove that the limiting distribution of these values is semi-circular  (i.e. $SU(2)$), confirming the usual Sato-Tate distribution in this setting.
	\end{abstract}
	\maketitle
	\section{Introduction and Statement of results}
	In the 1980's, Greene \cite{GreenePhD, greene} introduced Gaussian hypergeometric functions over finite fields, which are viewed as finite field analogues of classical hypergeometric series. These  functions have played central roles in the study of various objects in arithmetic geometry and number theory, such as  elliptic curves, the Eichler--Selberg trace formula, $K3$ surfaces, Calabi-Yau threefolds (for example, see \cite{ahlgren-ono, AO, ono-3, fop, Fuselier, FuselierHecke, Lennon1, Lennon, Long, McCarthy, McCarthyPJM, McCarthyRIMS, MOS, MP, ono, ono-book, PujahariSaikia, Rouse, TY, Z}), to name a few.
	To define them, suppose that  $A_1, A_2, \ldots, A_n$ and $B_1, B_2, \ldots, B_{n-1}$ are multiplicative characters of the finite field $\F_q.$ Then Greene \cite{GreenePhD, greene} defined the Gaussian hypergeometric function,
	\begin{align}
		{_{n}F_{n-1}}\left(\begin{array}{cccc}
			A_1, & A_2,& \ldots, & A_n\\
			~& B_1, &\ldots, &B_{n-1}
		\end{array}\mid x\right)_q:=\frac{q}{q-1}\sum_{\chi}{A_1\chi\choose\chi}{A_2\chi\choose B_1\chi}\cdots{A_n\chi\choose B_{n-1}\chi}\chi(x),\notag
	\end{align}
	where the summation runs over the multiplicative characters\footnote{For characters $\chi$, we have that $\chi(0):=0.$} of $\F_q^{\times},$ and where  ${A\choose B}$ is the normalized Jacobi sum $J(A,\overline{B})$, defined by
	\begin{align}\label{binomial}
		{A\choose B}:=\frac{B(-1)}{q}J(A,\overline{B}):=\frac{B(-1)}{q}\sum\limits_{x\in\F_q}A(x)\overline{B}(1-x).
	\end{align}

	In \cite{KHN}, some of the authors initiated a study of value distributions of families of Gaussian hypergeometric functions
	${_2F_1}\left(\begin{matrix}
		\phi, & \phi\\
		~& \varepsilon
	\end{matrix}\mid x\right)_q$ and ${_3F_2}\left(\begin{matrix}
		\phi, & \phi,& \phi\\
		~& \varepsilon, & \varepsilon
	\end{matrix}\mid x\right)_q,$ where $\phi$ is the quadratic character and $\varepsilon$ is the trivial character of $\F_q.$ They showed that the values of the first family converge to a semicircular distribution over large finite fields (i.e. $SU(2)$), whereas the latter converges to the so-called  {\it Batman distribution} (i.e.  $O(3)$). Such questions on the distributions of hypergeometric functions appear in the seminal work \cite{Katz,Katz2} of Katz where he deeply studies the geometric structures that underlie these functions.
	
	Here we extend this work to the Hessian Gaussian hypergeometric function values, where
	the entries involve cubic and trivial characters. To be precise, we investigate the following values:
	\begin{equation}
		{_2F_1}(\lambda)_q:={_2F_1}\left(\begin{matrix}
			\psi_3, & \overline{\psi_3}\\
			~& \varepsilon
		\end{matrix}\mid \lambda^3\right)_q,
	\end{equation}
	where $\psi_3$ is a  cubic character of $\F_q.$ Our first result estimates the moments of these values.	
	

	\begin{theorem}\label{2F1-moments}
		If $m$ and $r$ are  positive integers with $p^r\equiv1\pmod{3},$ then as  $p\rightarrow\infty$ we have
		
		$$
		p^{r(m/2-1)}\sum\limits_{\lambda\in\F_{p^r}} {_2F_1(\lambda)_{p^r}^m}=
		\begin{cases}
			\frac{(2n)!}{n!(n+1)!}+o_{m,r}(1), & \ \text{if} \ m=2n\ \text{is\ even} \\
			o_{m,r}(1), & \ \text{if\ m\ is\ odd}.
		\end{cases}
		$$
	\end{theorem}

	As a consequence of these moments, we obtain the limiting distribution of the normalized values $\sqrt{p^r}{_2F_1}(\lambda)_{p^r}\in[-2,2]$ as $p\rightarrow\infty.$

	\begin{corollary}\label{Distribution2F1}
		If $-2\leq a<b\leq 2$, and $r$ is a fixed positive integer, then 
		$$
		\lim\limits_{\substack{p\to\infty \\ p^r\equiv 1\pmod 3}}\frac{|\left\{\lambda\in\F_{p^r}: \sqrt{p^r}\cdot{_2F_1}(\lambda)_{p^r} \in [a,b]\right\}|}{p^r}=\frac{1}{2\pi}\int_{a}^b \sqrt{4-t^2} dt.
		$$
	\end{corollary}

	Our interest in the values arises from their connection to the Hessian elliptic curves. To make this precise, if $\lambda\in\F_q$ such that $\lambda^3\neq 27$ denote by $ E_\lambda^{\hes}$ the Hessian elliptic curve
	$$
	E_\lambda^{\hes}:\ \ \ \  x^3+y^3+1=\lambda xy,
	$$ 
	If $a_q^{\hes}(\lambda):=q+1-| E_\lambda^{\hes}(\F_q)|$ denotes the Frobenius trace of $E_{\lambda}^{\hes}$ over the finite field $\F_q$ and  $\lambda^3\neq0,27$, then (see Theorem \ref{Bridge}) 
	$$q\cdot{_2F_1}(3/\lambda)_q=-a_q^{\hes}(\lambda).$$ Therefore, we can rewrite Theorem \ref{2F1-moments} in terms of the trace of Frobenius of the Hessian form of elliptic curves whenever $q\equiv1\pmod{3}.$ This congruence condition is necessary in Theorem \ref{2F1-moments} for the existence of characters of order 3 which are entries of the Gaussian hypergeometric functions $_2F_1(\lambda)_{p^r}$. However, for $q\equiv2\pmod{3},$ the Frobenius traces of Hessian elliptic curves can be shown to  asymptotically have the same moments as the ${_2F_1}(\cdot)_{p^r}$ values.
	
	\begin{theorem}\label{MomentsOfTrace1}
		If $m, r$ are positive integers, $p$ is an odd prime, then as  $p\rightarrow\infty$ we have
		
		$$
		\sum\limits_{\substack{\lambda\in\F_{p^r}\\ \lambda^3\neq27}} a_{p^r}^{\hes}(\lambda)^{m}=
		\begin{cases}
			\frac{(2n)!}{n!(n+1)!}p^{r(n+1)}+o_{m,r}(p^{r(n+1)}), & \ \text{if} \ m=2n\ \text{is\ even} \\
			o_{ m,r}(p^{r(m/2+1)}) & \ \text{if\ m\ is\ odd}.
		\end{cases}
		$$
	\end{theorem}
	
	\begin{remark}
		There exist different variations and generalizations of the ${_{n+1}F_n}$-hypergeometric functions defined by Greene. For example, in their work on hypergeometric motives, Beukers, Cohen and Mellit \cite{BCM} define the function $H_q[{\alpha},{\beta} \ | \ t]$ (see Definition 1.1 of \cite{BCM}) as a variation of Greene's hypergeometric function. These generalize the functions ${_2F_1}(\lambda)_q$ we consider as they are defined for all $q=p^r,$ with $p>3.$ These functions resemble those studied by Katz \cite{Katz} more closely. In Corollary 1.7 of \cite{BCM} the authors expressed $a_q^{\hes}(\lambda)$ in terms of of the values of $H_q\left[(1/3,2/3),(1,1)\ |\ \cdot\right].$ Therefore, Theorem \ref{2F1-moments} can be stated in terms of the function $H_q.$
		
	\end{remark}

	\noindent As a corollary we find that the limiting distribution of the values of $p^{-r/2}\cdot a_{q}^{\hes}(\lambda)\in[-2,2]$ is semicircular, the usual {\it Sato-Tate distribution.}
	
	\begin{corollary}\label{TraceDistribution}
		If $-2\leq a<b\leq 2$ and $r$ is a fixed positive integer, then 
		$$
		\lim_{p\to\infty}\frac{|\left\{\lambda\in\F_{p^r}: \lambda^3\neq27 \ \text{and}\ p^{-r/2}\cdot a_{q}^{\hes}(\lambda) \in [a,b]\right\}|}{p^r}=\frac{1}{2\pi}\int_{a}^b \sqrt{4-t^2} dt.
		$$
	\end{corollary}
	
	\begin{remark}
		Corollary \ref{TraceDistribution} is a refinement of a classical theorem of Birch \cite{birch}, which asserts that the semicircular distribution is the limiting behavior among all elliptic curves over finite fields.  Corollary \ref{TraceDistribution} is the restriction of his result to the Hessian family of elliptic curves. In \cite{KP}, Kane and the second author gave another type of refinement,  where one restricts the Frobenius traces of elliptic curves to arithmetic progressions. 
		These results show the universality of the Sato--Tate distribution, which was proven \cite{SatoTateProof} in the much more challenging setting of a fixed non-CM elliptic curve, with the primes varying, by Clozel, Harris, Shepherd-Barron and Taylor in 2008.
	\end{remark}

	
	This paper is organized as follows. In Section 2, we discuss some important facts about Hessian elliptic curves. In particular, we recall the connection between the values $_2F_1(\lambda)_q$ and the Frobenius traces of these curves. Furthermore, we recall work of Schoof that gives class numbers formulas for the number of their isomorphism classes over $\F_q.$  The main result of that section gives explicit formulas for the moments of $a_q^{\hes}(\lambda)$ and $_2F_1(\lambda)_q$ in terms of weighted sums of class numbers. In Section 3, we recall specific facts of harmonic Maass forms and obtain some combinatorial identities that will be useful for estimating these moment formulas. In Section 4, we establish asymptotics for these moments, and we conclude with the proofs of our main results.

	
	\section{The Arithmetic of $E_{\lambda}^{\hes}$ and $_2F_1(\lambda)_q$}
	In this section we briefly discuss some important results about Hessian elliptic curves. Moreover, we discuss how the values of  $_2F_1(\lambda)_q$ hypergeometric functions are related to the Frobenius trace of Hessian elliptic curves.
	We use this relation to obtain the moment formulas of $_2F_1(\lambda)_q.$ To this end, we state the following result, which is a reformulation of Theorem 3.5 of \cite{B-K}.
	
	\begin{theorem}\label{Bridge}
		Let $p>3$ be a prime and $q=p^r$ such that $q\equiv1\pmod{3}.$ If $\lambda\in\F_q$ such that $\lambda^3\neq27,$ then we have that
		$$
		q\cdot {_2F_1}\left(\frac{3}{\lambda}\right)_q=-a_q^{\hes}(\lambda).
		$$
	\end{theorem}
	\begin{proof}
		It follows immediately from Theorem 3.5 of \cite{B-K} and counting the points $(1:-\omega:0)$ with $\omega^3=1$ at infinity in the projective space. 
	\end{proof}
	
	\subsection{Hessian form of elliptic curves}
	We now recall certain facts about the Hessian elliptic curves $E_{\lambda}^{\hes}$ over $\F_q.$

	\begin{proposition}\label{count}\cite{M-W} 
		Let $K$ be a field with char$(K)\neq2,3,$ and suppose that $\lambda^3\neq27.$ Then the following are true:
		
		\noindent
		(1)  We have that the $j$-invariant of $E_{\lambda}^{\hes}$ is given by
		$$
		j(E_{\lambda}^{\hes})=\frac{\lambda^3(\lambda^3+216)^3}{(\lambda^3-27)^3}.
		$$
		
		\noindent
		(2)
		We have that $j(E_{\lambda}^{\hes})=0$ if and only if $\lambda(\lambda^3+6^3)=0.$ 
		
		\noindent
		(3) We have that  $j(E_{\lambda}^{\hes})=1728$ if and only if $\lambda^6-540\lambda^3-(18)^3=0.$
	\end{proposition}

	\noindent For $\lambda\in\F_q$ such that $\lambda^3\neq27,$ we define
	$$
	\I:=\{v\in\F_q: \ E_\lambda^{\hes}\cong_{\F_q}E_v^{\hes}\}.
	$$
	The next two lemmas count the number of elements of $\I.$
	\begin{lemma}\label{cardinality-1}\cite[Lemma 5]{F-R}
		Let $p>3$ be a prime and $q=p^r$ such that $q\equiv1\pmod{3}.$ Then for $\lambda\in\F_q$ such that $\lambda^3\neq27$ we have
		$$
		|\I|=\begin{cases}
			12 &\ \text{if}\ A_\lambda B_\lambda\neq0\\
			4 &\ \text{if}\ A_\lambda=0\\
			6 &\ \text{if}\ B_\lambda=0,
		\end{cases}
		$$
		where $A_\lambda=\lambda(\lambda^3+6^3)$ and $B_\lambda=\lambda^6-540\lambda^3-(18)^3.$
	\end{lemma}

	\begin{lemma}\label{cardinality-2}\cite[Lemma 5]{F-R}
		Let $p>3$ be a prime and $q=p^r$ such that $q\equiv2\pmod{3}.$ Then for $\lambda\in\F_q$ such that $\lambda^3\neq27$ we have
		$$
		|\I|=1.
		$$
		
		\begin{lemma}\label{j-invariant-0,1728}
			Suppose that $q\equiv2\pmod{3}$. Then the following holds:\\
			
			\noindent (1) If $j(E_{\lambda}^{\hes})=0,$ then $a_q^{\hes}(\lambda)=0.$
			
			\noindent (2) There are no $E_{\lambda}^{\hes}$ such that $j(E_{\lambda}^{\hes})=1728.$
			
		\end{lemma}
		
		\begin{proof}
			Claim (1) follows from \cite[pg. 304]{rosen} and claim (2) follows from the unsolvability of $\lambda^6-540\lambda^3-(18)^3=0.$
		\end{proof}
	\end{lemma}
	
	\noindent We now recall the following two facts about the group of $\F_q$-rational points of the Hessian elliptic curves $E_{\lambda}^{\hes}$ and about elliptic curves with prescribed subgroups $\Z/3\Z\times\Z/3\Z$ and $\Z/3\Z$.

	\begin{proposition}\label{isomorphism-1}\cite[Theorem 3.6]{M-W}
		Let $q\equiv1\pmod{3}$ and $E/\F_q$ be an elliptic curve. Then $E(\F_q)$ contains the subgroup $\Z/3\Z\times\Z/3\Z$ if and only if $E$ is isomorphic over $\F_q$ to a Hessian curve $$E_\lambda^{\hes}:\ \ \ x^3+y^3+1=\lambda xy$$ for some $\lambda\in\F_q,\lambda^3\neq27.$
	\end{proposition}
	
	\begin{proposition}\label{isomorphism-2}\cite[Theorem 3.2]{M-W}
		Let $q\equiv2\pmod{3}$ and $E/\F_q$ be an elliptic curve. Then $E(\F_q)$ has a point of order 3 if and only if $E$ is isomorphic over $\F_q$ to a Hessian curve $$E_\lambda^{\hes}:\ \ \ x^3+y^3+1=\lambda xy$$ for some $\lambda\in\F_q,\lambda^3\neq27$.
	\end{proposition}

	\subsection{Isomorphism classes of elliptic curves and Hurwitz class numbers}
	Here, we provide formulas to count the number of isomorphism classes of elliptic curves over $\F_q.$ We recall a theorem by the work of Schoof \cite{schoof} that gives this number in terms of Hurwitz class number.
	To do this, we first recall some basic notation. 
	If $-D<0$ such that $-D\equiv0,1\pmod{4},$ then $\mathcal{O}(-D)$ denotes the unique imaginary quadratic order with discriminant $-D.$ Let $h(\mathcal{O}(-D))=h(\mathcal{O})$ denote\footnote{We note that $H(D)=H^*(D)=h(D)=0$ whenever $-D$ is neither zero nor a negative discriminant.} the order of the class group of $\mathcal{O}(-D)$ and let $\omega(\mathcal{O}(-D))=\omega(\mathcal{O})$ denote half the number of roots of unity in $\mathcal{O}(-D).$  In this notation, we define the Hurwitz class numbers
	\begin{equation}
		H(D):=\sum\limits_{\mathcal{O}\subseteq\mathcal{O'}\subseteq\mathcal{O}_{\text{max}}}h(\mathcal{O'})
		\ \ \ {\text {\rm and}}\ \ \ 
		\ \ H^{\ast}(D):=\sum\limits_{\mathcal{O}\subseteq\mathcal{O'}\subseteq\mathcal{O}_{\text{max}}}\frac{h(\mathcal{O'})}{\omega(\mathcal{O'})},
	\end{equation}
	where the sum is over all orders $\mathcal{O'}$ between $\mathcal{O}$ and the maximal order $\mathcal{O}_{\text{max}}.$
	Here note that we have $H(D)=H^{\ast}(D)$ unless $\mathcal{O}_{\text{max}}=\Z[i]$ or $\Z\left[\frac{-1+\sqrt{-3}}{2}\right],$ where the terms corresponding to $\mathcal{O}_{\text{max}}$ differs by a factor of $2$ and $3$ respectively.

	The following theorem of Schoof \cite{schoof} counts isomorphism classes of elliptic curves with prescribed subgroups in terms of Hurwitz class numbers.

	\begin{theorem}[Section 4 of \cite{schoof}]\label{Schoof} If $p\geq 5$ is prime, and  $q=p^r,$ then the following are true.

		\noindent
		(1) If $n\geq 2$ and  $s$ is a nonzero integer for which  $p|s$ and $s^2\neq 4q,$  then there are no elliptic curves $E/\F_q$
		with $|E(\F_q)|=q+1-s$ and 
		$ \Z/ n\Z\times \Z/ n\Z\subseteq E(\F_q).$
		
		\noindent
		(2)  If $r$ is even and $s=\pm 2p^{r/2},$ then the number of isomorphism classes of elliptic curves over $\F_q$ with
		$\Z/ n\Z\times \Z/ n\Z \subseteq E(\F_q)$ and
		$|E(\F_q)|=q+1-s$ is
		\begin{equation}\label{Ap}
			\frac{1}{12}\left(p+6-4\leg{-3}{p}-3\leg{-4}{p}\right),
		\end{equation}
		where $\leg{\cdot}{p}$ is the Legendre symbol.
		
		\noindent 
		(3) If $r$ is even, $s=\pm p^{r/2}$ and $p\not\equiv1\pmod{3},$ then the number of isomorphism classes of elliptic curves over $\F_q$ with
		$\Z/ n\Z \subseteq E(\F_q)$ and
		$|E(\F_q)|=q+1-s$ is
		\begin{equation}\label{Ap-2}
			1-\leg{-3}{p}.
		\end{equation}
		
		\noindent
		(4)  Suppose that $n$ and $s$ are integers such that $s^2\leq 4q,$ $p\nmid s,$ $n^2\mid (q+1-s),$ and $n\mid (q-1).$ Then the number of isomorphism classes of elliptic curves over $\F_q$ with
		$|E(\F_q)|=q+1-s$ and $\Z/ n\Z\times \Z/ n\Z\subseteq E(\F_q)$ is $H\left(\frac{4q-s^2}{n^2}\right).$

	\end{theorem}

	\subsection{Formulas for Frobenius trace moments and $_2F_1$ moments.}
	In this section we make use of the results of the previous section to derive the moment formulas for Frobenius trace of Hessian curves and the $_2F_1$ functions.

	\begin{proposition}\label{Trace-moment-class-number}
		Let $p\geq5$ be a prime. If $r$ and $m$ are positive integers, then the following are true for $q=p^r,$ where in each sum we have $-2\sqrt{q}\leq s\leq 2\sqrt{q}.$
		
		\noindent
		(1) If $r$ is odd and $p^r\equiv1\pmod{3}$, then we have 
		$$
		\sum\limits_{\substack{\lambda\in\F_q\\ \lambda^3\neq27}} a_{p^r}^{\hes}(\lambda)^m=12\sum\limits_{\substack{p\nmid s\\s\equiv q+1\pmod9}}H^*\left(\frac{4q-s^2}{9}\right)s^m.
		$$
		
		\noindent 
		(2) If $r$ is even and $p^r\equiv1\pmod{3}$, then there is a rational constant $R(q)\in[0,12]$ such that 
		
		$$
		\sum\limits_{\substack{\lambda\in\F_q\\ \lambda^3\neq27}} a_{p^r}^{\hes}(\lambda)^m=(-1)^mR(q)A(q)q^{m/2}+12\sum\limits_{\substack{p\nmid s\\s\equiv q+1\pmod9}}
		H^*\left(\frac{4q-s^2}{9}\right)s^m,
		$$
		where $A(q)=\frac{1}{12}\left(p+6-4\leg{-3}{p}-3\leg{-4}{p}\right)$ and the constant $R(q)$ can be described theoretically as the number of Hessian form elliptic curves present in a $\F_q$-isomorphism class with trace $a_q^{\hes}(\lambda)=\pm2\cdot \sqrt{p^r}.$
		
		\noindent
		(3) If $p^r\equiv2\pmod{3},$ then we have 
		$$
		\sum\limits_{\substack{\lambda\in\F_q\\ \lambda^3\neq27}} a_{p^r}^{\hes}(\lambda)^m=\sum\limits_{\substack{p\nmid s\\s\equiv q+1\pmod3}}H^*(4q-s^2)s^m.
		$$
	\end{proposition}
	
	\begin{proof}
		
		Here we prove claim (1).
		
		To this end, we first rewrite the sum as a sum over all the possible traces.  By Propostion~\ref{isomorphism-1} and (1) of Theorem~\ref{Schoof}, we have that
		$$
		\sum\limits_{\substack{\lambda\in\F_q \\ \lambda^3\neq 27}}a_{p^r}^{\hes}(\lambda)^m = \sum\limits_{\substack{s\equiv q+1\pmod 9 \\ (s,p)=1}} s^m \cdot\#\{\lambda\in\F_{q}, a_\lambda^{\hes}(q)=s\}
		$$
		Furthermore, by Proposition~\ref{isomorphism-1}, $\#\{\lambda\in\F_{q}, a_\lambda^{\hes}(q)=s\}$ is itself a sum of contributions that arise from each isomorphism class of elliptic curves $E$ with $\#E(\F_q)=q+1-s$ and $\Z/3\Z\times\Z/3\Z\subset E(\F_q).$ To count the number of Hessian elliptic curves in each isomorphism class, we consider three possible cases for the endomorphism ring. 
		
		First, suppose that  $\mathcal{O}\left(\frac{4q-s^2}{9}\right) \not\subseteq\Z[i],\Z\left[\frac{-1+\sqrt{-3}}{2}\right].$ This implies that $j(E_\lambda^{\hes})\neq0,1728$ (see \cite[Sec. 3]{schoof}). Therefore, Proposition~\ref{j-invariant-0,1728}, Lemma~\ref{cardinality-1} and Proposition~\ref{isomorphism-1} imply that each isomorphism class of elliptic curves $E$ with $\#E(\F_q)=q+1-s$ corresponds to one set $I_\lambda^{\hes}$ with $12$ elements. By Theorem \ref{Schoof} there are $H\left(\frac{4q-s^2}{9}\right)$ such isomorphism classes, and therefore, the contribution to the sum is
		$$
		12 H\left(\frac{4q-s^2}{9}\right)s^m=12 H^{*}\left(\frac{4q-s^2}{9}\right)s^m
		$$
		
		Next suppose that $s\equiv q+1\pmod{9}$ and $\mathcal{O}\left(\frac{4q-s^2}{9}\right) \subseteq\Z[i].$ Similarly to the previous case there are $H\left(\frac{4q-s^2}{9}\right)$ isomorphism classes with $q+1-s$ points and containing the subgroup 
		$\Z/ 3\Z\times\Z/ 3\Z.$ Again by Proposition \ref{count} and Lemma \ref{cardinality-1} we have that all such classes except one correspond to $\I$ with 12 elements and the last one corresponds to the set of 6 elements with $j=1728.$ Recall that $H$ and $H^*$ differ by 1/2 in this case. Hence, the contribution of this case to the required sum is given by 
		
		$$12 H^{*}\left(\frac{4q-s^2}{9}\right)s^m.$$
		
		Finally, suppose that $s\equiv q+1\pmod{9}$ and $\mathcal{O}\left(\frac{4q-s^2}{9}\right)\subseteq \Z\left[\frac{-1+\sqrt{-3}}{2}\right].$ By arguing as in the previous two cases we conclude that the contribution of this case to the required sum is given by 
		
		$$
		12 H^{*}\left(\frac{4q-s^2}{9}\right)s^m.
		$$
		
		Claim (2) can be settled analogously and shall be left to the reader. This claim requires consideration of the supersingular cases where $s=a_{q}^{\hes}(\lambda)=\pm2\sqrt{q}.$ To prove Claim (3) we apply similar arguments as discuss above and use Proposition \ref{isomorphism-2}, Lemma \ref{cardinality-2} and Lemma \ref{j-invariant-0,1728}.

	\end{proof}

	\begin{proposition}\label{2F1-moment-class-number}
		Let $p\geq5$ be a prime. If $r$ and $m$ are positive integers, then the following are true for $q=p^r,$ where all the sums on the right sides run over $s\in\Z$  such that $-2\sqrt{q}\leq s\leq 2\sqrt{q}.$
		
		\noindent
		(1) If $r$ is odd, then we have 
		$$
		(-q)^m\sum\limits_{\lambda\in\F_q}{_2F_1}(\lambda)_q^m=1- a_{q}^{\hes}(0)^m+12
		\sum\limits_{\substack{p\nmid s\\s\equiv q+1\pmod9}}
		H^*\left(\frac{4q-s^2}{9}\right)s^m.
		$$
		
		\noindent
		(2) If $r$ is even, then we have
		
		$$
		(-q)^m\sum\limits_{\lambda\in\F_q}{_2F_1}(\lambda)_q^m=1- a_{q}^{\hes}(0)^m+R(q)A(q)q^{m/2}+12
		\sum\limits_{\substack{p\nmid s\\s\equiv q+1\pmod9}}
		H^*\left(\frac{4q-s^2}{9}\right)s^m,
		$$
		where $A(q)$ and $R(q)$ are same as defined in Proposition \ref{Trace-moment-class-number}.
	\end{proposition}

	\begin{proof}
		By transforming $\lambda\mapsto3/\lambda$ and using Theorem \ref{Bridge} we can rewrite the sum in terms of $a_q^{\hes}(\lambda)$ as follows: 
		$$
		q^m\sum\limits_{\lambda\in\F_q}{_2F_1}(\lambda)_q^m=q^m\cdot {_2F_1}(1)_q^m-(-1)^m a_{q}^{\hes}(0)^m+
		(-1)^m\sum\limits_{\substack{\lambda\in\F_q\\ \lambda^3\neq27}}a_q^{\hes}(\lambda)^m.
		$$
		Now, by making use of the value $_2F_1(1)_q=-1/q$ and Proposition \ref{Trace-moment-class-number} we conclude the proof.
		
	\end{proof}
	
	\section{Harmonic Maass forms and Mock modular forms}\label{HarmonicMaassForms}
	
	In this section we obtain the asymptotic formulas for the weighted sums of class numbers that we have obtained in the last section. To do this we need the theory of harmonic Maass forms (for a detailed study of harmonic Maass forms,  see \cite{BFOR}). Let $\Gamma(\alpha;x):=\int_{\alpha}^{\infty}e^{-t}t^{x-1}dt$ be the incomplete Gamma function. We begin with the following celebrated theorem of Zagier.
	
	\begin{theorem}[\cite{Zagier1}]\label{ZagierSeries}
		The function
		$$
		\mathcal{H}(\tau)=-\frac{1}{12}+\sum\limits_{n=1}^\infty H^\ast(n)q_\tau^n+\frac{1}{8\pi\sqrt{y}}+\frac{1}{4\sqrt{\pi}}\sum\limits_{n=1}^\infty n\Gamma(-\frac{1}{2}; 4\pi n^2y)q_{\tau}^{-n^2},
		$$
		where $\tau=x+iy\in \HH$ and $q_\tau:=e^{2\pi i\tau},$ is a weight $3/2$ harmonic Maass form with manageable growth on $\Gamma_0(4).$ 
	\end{theorem}
	
	In general, note that every weight $k\neq 1$ harmonic weak Maass form $f(\tau)$ has a Fourier expansion of the form
	\begin{equation}\label{HMFFourier}
		f(\tau)=f^{+}(\tau)+\frac{(4\pi y)^{1-k}}{k-1}\overline{c_f^{-}(0)}+f^{-}(\tau),
	\end{equation}
	where 
	\begin{equation}\label{HMSParts}
		f^{+}(\tau)=\sum\limits_{n=m_0}^\infty c_f^{+}(n)q_\tau^n \ \ \
		{\text {\rm and}}\ \ \ 
		f^{-}(\tau)=\sum\limits_{\substack{n=n_0\\ n\neq 0}}^\infty \overline{c_f^{-}(n)}n^{k-1}\Gamma(1-k;4\pi |n| y)q_\tau^{-n}.
	\end{equation}
	
	The function $f^{+}(\tau)$ is called the {\it holomorphic part} of $f,$ or its corresponding {\it mock modular form.} In this work we consider $\mathcal{H}(9\tau),$ which is a weight $3/2$ harmonic Maass form with manageable growth on $\Gamma_0(36).$

	To obtain the weighted class number sums in Proposition~\ref{Trace-moment-class-number} as Fourier coefficients of harmonic Maass forms, we need combinatorial operators involving derivatives of modular forms. To this end, we recall the Rankin-Cohen bracket operators.
	Let $f$ and $g$ be smooth functions defined on the upper-half complex plane $\HH$, and let $k, l\in\R_{>0}$ and $\nu\in\N_0.$ Then the $\nu$th Rankin-Cohen bracket of $f$ and $g$ is defined by
	\begin{equation}\label{RankinCohenBracket}
		[f,g]_\nu:=\frac{1}{(2\pi i)^\nu}\sum\limits_{r+s=\nu}(-1)^r\binom{k+\nu-1}{s}\binom{l+\nu-1}{r}\frac{d^r}{d\tau^r}f\cdot\frac{d^s}{d\tau^s}g.
	\end{equation}
	Note that these operators preserve modularity.
	
	\begin{proposition}[Theorem 7.1 of \cite{cohen}]\label{BracketProposition}
		Let $f$ and $g$ be (not necessarily holomorphic) modular forms of weights $k$ and $l,$ respectively on a congruence subgroup $\Gamma.$ Then the following are true.

		\noindent
		(1) We have that $[f,g]_\nu$ is modular of weight $k+l+2\nu$ on $\Gamma.$ 
		
		\noindent
		(2) If $\gamma\in SL_2(\R),$ then
		under the usual modular slash operator we have
		$$
		[f|_k\gamma,g|_l\gamma]_\nu=([f,g]_\nu)|_{k+l+2\nu}\gamma.
		$$
	\end{proposition}

	The weighted sums of class numbers that appear in the moment formulas of the previous section can be constructed from Fourier coefficients of certain nonholomorphic modular forms by implementing the Rankin--Cohen bracket operators on the Zagier's function $\mathcal{H}(9\tau)$ and certain univariate theta functions. This method was already treated in \cite{BKP2},   \cite{KP}, \cite{mertens} and \cite{KHN}. In order to apply facts from the theory of holomorphic modular forms to these brackets, we recall a standard method known as holomorphic projection and some facts about it.
	
	Let $f:\HH\rightarrow\C$ be a (not necessarily holomorphic) modular form of weight $k\geq 2$ on a congruence subgroup $\Gamma$ with Fourier expansion
	$$
	f(\tau)=\sum_{n\in\Z}c_f(n,y)q_{\tau}^n,
	$$
	where $\tau=x+iy.$ Let $\{\kappa_1,\ldots, \kappa_M\}$ be the cusps of $\Gamma,$ where  $\kappa_1:=i\infty.$ Moreover, for each $j$ let $\gamma_j\in \SL_2(\Z)$ satisfy
	$\gamma_j\kappa_j=i\infty.$ Then suppose the following are true.

	\noindent
	(1) There is an $\varepsilon>0$ and a constant $c_0^{(j)}\in\C$ for which
	$$
	f\left(\gamma_j^{-1}w\right)\left(\frac{d\tau}{dw}\right)^{k/2}=c_0^{(j)}+O(\im(w))^{-\varepsilon},
	$$
	for all $j=1,\ldots,M$ and $w=\gamma_j\tau.$

	\noindent
	(2) For all $n>0,$ we have that $c_f(n,y)=O(y^{2-k})$ as $y\rightarrow0.$
	
	Then the {\it holomorphic projection of $f$} is defined by
	\begin{equation}
		(\pi_{\text{hol}}f)(\tau):=c_0+\sum\limits_{n=1}^{\infty}c(n)q_{\tau}^n,
	\end{equation}
	where $c_0=c_0^{(1)}$ and for $n\geq1$
	$$
	c(n)=\frac{(4\pi n)^{k-1}}{(k-2)!}\int_{0}^{\infty}c_f(n,y)e^{-4\pi ny}y^{k-2}dy.
	$$
	
	\begin{proposition}[Proposition 10.2 of \cite{BFOR}]\label{HolProjProp}
		If the above hypotheses are true, then for $k>2$ $\pi_{\text{hol}}(f)$ is a weight $k$ holomorphic modular form on $\Gamma.$ 
	\end{proposition}
	
	If $f$ is a harmonic Maass form of weight $k\in \frac{1}{2}\Z$ on $\Gamma_0(N)$ with manageable growth at the cusps and $g$ is a holomorphic modular form of weight $l$ on $\Gamma_0(N),$ then the holomorphic modular form obtained by Proposition~\ref{HolProjProp} has the following decomposition form 
	\begin{equation}\label{HolomorphicProjectionDecomposition}
		\pi_{\text{hol}}([f,g]_\nu)=[f^{+},g]_\nu+\frac{(4\pi)^{1-k}}{k-1}\overline{c_f^{-}(0)}\pi_{\text{hol}}([y^{1-k},g]_\nu)+\pi_{\text{hol}}([f^{-},g]_\nu).
	\end{equation}
	
	The two following computations of Mertens \cite{mertens} provide closed formulas for the Fourier expansions of the second and third terms of the above decomposition.
	
	\begin{lemma}[Lemma V.1.4 of \cite{mertens}]\label{HolProjExplicit1}
		If the hypotheses above is true and $g(\tau)$ has Fourier expansion
		$g(\tau)=\sum\limits_{n=0}^{\infty}a_g(n)q_{\tau}^n,$
		then we have
		$$
		\frac{(4\pi)^{1-k}}{k-1}\pi_{\text{hol}}([y^{1-k},g]_\nu)=\kappa(k,l,\nu)\cdot \sum\limits_{n=0}^\infty n^{k+\nu-1}a_g(n)q_\tau^n,
		$$
		where
		$$
		\kappa(k,l,\nu):=\frac{1}{(k+l+2\nu-2)!(k-1)}\sum\limits_{\mu=0}^\nu\left(\frac{\Gamma(2-k)\Gamma(l+2\nu-\mu)}{\Gamma(2-k-\mu)}\binom{k+\nu-1}{\nu-\mu}\binom{l+\nu-1}{\mu}\right).
		$$
	\end{lemma}
	
	\noindent  To state the explicit Fourier expansion of the third term, let $P_{a,b}(X,Y)$ be the homogeneous polynomial of degree $a-2$ defined by 
	\begin{align}\label{1}
		P_{a,b}(X,Y):=\sum\limits_{j=0}^{a-2}\binom{j+b-2}{j}X^j(X+Y)^{a-j-2},
	\end{align}
	where $a\geq 2$ is a positive integer and $b$ is any real number. 
	
	\begin{theorem}[Theorem V.1.5 of \cite{mertens}]\label{HolProjExplicit2}
		If  $c_f^{-}(n)$ and $a_g(n)$ are bounded polynomially, then we have
		$\pi_{\text{hol}}([f^{-},g]_\nu)=\sum\limits_{r=1}^\infty b(r)q_\tau^{r},$
		where 
		
		\begin{displaymath}
			\begin{split}
				b(r)=-\Gamma(1-k)&\sum\limits_{m-n=r}a_g(m)\overline{c^{-}_f(n)}\sum\limits_{\mu=0}^\nu\binom{k+\nu-1}{\nu-\mu}\binom{l+\nu-1}{\mu}m^{\nu-\mu}\\ 
				&\ \ \ \ \ \ \ \ \ \ \ \ \ \ \ \ \ \ \ \ \times \left(m^{\mu-2\nu-l+1}P_{k+l+2\nu,2-k-\mu}(r,n)-n^{k+\mu-1}\right),
			\end{split}
		\end{displaymath}
		where the sum runs over positive integers $m,n.$
	\end{theorem}

	\subsection{Combinatorial Identities}
	In this section we derive certain combinatorial formulas that are used in the proof of the main results. To this end, we recall some basic combinatorial definitions.

	If $\alpha\in\C,$ and $j\geq 0,$ define the Pochhammer symbol 
	$$
	(\alpha)_j:=\begin{cases}
		1 &\ \ \text{ if }j=0 \\
		\alpha(\alpha+1)\cdots(\alpha+j-1) &\ \ \text{ if }j>0.
	\end{cases}
	$$
	Furthermore, we recall that Euler's $\Gamma$ function satisfies the following properties:
	\begin{enumerate}
		\item $\Gamma(z+1)=z\Gamma(z).$
		\item $\Gamma(\frac{1}{2})=\sqrt{\pi}.$
		\item $	\Gamma(z)\Gamma(z+1/2)=2^{1-2z}\sqrt{\pi}\Gamma(2z).$ (Legendre's duplication formula, see \cite[Th. 1.5.1]{Askey})
	\end{enumerate}
	
	In this notation, we have the following identities.

	\begin{lemma}\label{CombLemma3Equation}
		If $ k\leq \nu$ are nonnegative integers, then we have
		$$
		\sum\limits_{\mu=0}^{\nu-k}{2\nu+1\choose 2\nu-2\mu+1}{\nu-\mu\choose k}=4^{\nu-k}\frac{(2\nu-k)!}{k!(2\nu-2k)!}.
		$$
	\end{lemma}
	\begin{proof}
		Let $C_\mu:={2\nu+1\choose 2\nu-2\mu+1}{\nu-\mu\choose k}=\frac{(2\nu+1)!}{(2\nu-2\mu+1)!(2\mu)!}
		\frac{(\nu-\mu)!}{k!(\nu-\mu-k)!}.$ Then we have
		$$
		\frac{C_{\mu+1}}{C_\mu}=\frac{(\mu-\nu-\frac{1}{2})(\mu-\nu+k)}{(\mu+1)(\mu+\frac{1}{2})}.
		$$
		
		As we can see, the quotient $\frac{C_{\mu+1}}{C_\mu}$ is a monic rational function in $\mu$, hence $\sum C_\mu$ a multiple of a hypergeometric series
		$$_2F_1^{\text{class}}\left(\begin{matrix}
			a, & b\\
			~& c
		\end{matrix}\mid z\right):=\sum_{j=0}^{\infty}\frac{(a)_j(b)_j}{(c)_j}~\frac{z^j}{j!}.$$
		
		Namely we have
		
		$$
		\sum\limits_{\mu=0}^{\nu-k}C_\mu={\nu\choose k}\cdot{_2F_1^{\text{class}}}\left(\begin{matrix}
			-\nu-\frac{1}{2}, & -\nu+k\\
			~ & \frac{1}{2}
		\end{matrix}\mid 1\right).
		$$
		By Gauss's identity (see (1.3) of \cite{bailey} and \cite{Gauss}), we have
		$$
		{_2F_1^{\text{class}}}\left(\begin{array}{cc}
			a & b\\
			~&c
		\end{array}|\ 1 \right)=\frac{\Gamma(c)\Gamma(c-a-b)}{\Gamma(c-a)\Gamma(c-b)},
		$$
		whenever $\re(c-a-b)>0,$ and therefore using this and Legendre's duplication formula we have
		
		$$
		\sum\limits_{\mu=0}^{\nu-k} C_\mu = 4^{\nu-k}\frac{(2\nu-k)!}{k!(2\nu-2k)!}.
		$$
	\end{proof}
	
	\begin{lemma}\label{lemma-odd-1}
		Let $\kappa(k,l,\nu)$ be defined as in Lemma \ref{HolProjExplicit1}. Then, for all 
		$\nu\geq0$
		$$
		\kappa\left(\frac{3}{2},\frac{3}{2},\nu\right)=\sqrt{\pi}\cdot 2^{-2\nu-1}{{2\nu+2}\choose{\nu+1}}.
		$$
		
	\end{lemma}
	\begin{proof}
		By definition, we have
		\begin{align}
			\kappa\left(\frac{3}{2},\frac{3}{2},\nu\right)&=\frac{2}{(2\nu+1)!}\sum\limits_{\mu=0}^{\nu}\left[\frac{\Gamma(1/2)\Gamma(2\nu-\mu+3/2)}{\Gamma(1/2-\mu)}{{\nu+1/2}\choose{\nu-\mu}}{{\nu+1/2}\choose{\mu}}\right]\notag\\
			&=\frac{2\sqrt{\pi}}{(2\nu+1)!}\sum\limits_{\mu=0}^{\nu} c_\mu,\notag
		\end{align}
		
		where
		$
		c_\mu=\frac{\Gamma(2\nu-\mu+3/2)}{\Gamma(1/2-\mu)}{{\nu+1/2}\choose{\nu-\mu}}{{\nu+1/2}\choose{\mu}}.
		$ It is easy to compute that
		$$
		\frac{c_{\mu+1}}{c_\mu}=\frac{(\mu+1/2)(\mu-\nu)(\mu-\nu-1/2)}{(\mu-2\nu-1/2)(\mu+3/2)(\mu+1)}.
		$$
		
		Therefore, we have
		$$
		\kappa\left(\frac{3}{2},\frac{3}{2},\nu\right)=\frac{2\Gamma(2\nu+3/2)}{(2\nu+1)!}{{\nu+1/2}\choose{\nu}}\cdot {_{3}F_{2}^{\text{class}}}\left(\begin{array}{cccc}
			1/2, & -\nu,& -\nu-1/2\\
			~& -2\nu-1/2, & 3/2
		\end{array}\mid 1\right).
		$$
		Now, by using some standard hypergeometric identities e.g. Pfaff-Schulz identity (see Theorem 2.2.6 of \cite{Askey}) and Legendre's Duplication Formula, we conclude the result. 
	\end{proof}


	\section{Asymptotics and the proofs of the main results}
	In this section, we compute the asymptotics of weighted class number sums that appear in Proposition~\ref{Trace-moment-class-number}. These calculations are analogous to those used in \cite{KP, KHN}. 
	To this end, we need results from the previous section, certain standard bounds of class numbers and bounds on Fourier coefficients of cusp forms. We first recall the celebrated theorem of Deligne, which bounds Fourier coefficients of integer weight cusp forms. 
	\begin{theorem}[Remark 9.3.15 of \cite{stromberg}]\label{Deligne}
		Let $f=\sum\limits_{n\geq 1} a(n)q_\tau^n$ be a cusp form of integer weight $k$ on a congruence subgroup. Then for all $\varepsilon>0$ we have
		$a(n)=O_\varepsilon(n^{(k-1)/2+\varepsilon}).$
	\end{theorem}
	
	\begin{lemma}\label{Bound-1}
		Let $r$ and $m$ be fixed positive integers. Then as $p\rightarrow\infty,$ we have
		$$
		\sum\limits_{s\in \Omega_{p^r}}H^*\left(\frac{4p^r-s^2}{9}\right)s^m=o_{m,r}(p^{r(m/2+1)})
		$$ and  
		$$\sum\limits_{s\in \Omega_{p^r}'}H^*(4p^r-s^2)s^m=o_{m,r}(p^{r(m/2+1)}),
		$$
		where $\Omega_{p^r}$ and $\Omega_{p^r}'$ are defined as follows:
		$$\Omega_{p^r}:=\{s\in[-2\sqrt{p^r}, 2\sqrt{p^r}]\cap \Z: p\mid s \ \text{and}\ s\equiv p^r+1\pmod{9}\}$$ and   $$\Omega_{p^r}':=\{s\in[-2\sqrt{p^r}, 2\sqrt{p^r}] \cap \Z: p\mid s \ \text{and}\ s\equiv p^r+1\pmod{3}\}.$$
		
	\end{lemma}
	\begin{proof}
		It is easy to see that there are at most $2p^{r/2-1}$ nonzero integers $s$ that are present in $\Omega_{p^r}$ and $\Omega_{p^r}'.$ Therefore, we have the following trivial bound:
		$$
		\sum\limits_{s\in \Omega_{p^r}}H^*\left(\frac{4p^r-s^2}{9}\right)s^m\leq 2p^{r/2-1}(2p^{r/2})^m\cdot \max \left\{H^*\left(\frac{4p^r-s^2}{9}\right)\right\}.
		$$
		The first claim follows easily from \cite[Lemma 2.2]{griffin-ono-tsai}. The proof of the second bound is analogous.

	\end{proof}

	\subsection{Asymptotic Moment formulas of $a_{q}^{\hes}(\lambda)$ when $q\equiv1\pmod{3}$}\label{EvenMoments2F1Asymptotics}


	We now compute asymptotics of weighted class number sums that appear in even moments of Frobenius traces when $q\equiv 1\pmod 3.$
	
	\begin{proposition}\label{AsymptoticsEven}
		If $n$ is a nonnegative integer and $r\geq 1$ is fixed, then as $p\rightarrow\infty$ with $q:=p^r\equiv 1\pmod 3$, we have
		$$
		12\sum\limits_{s\equiv q+1\pmod 9} H^\ast\left(\frac{4q-s^2}{9}\right)s^{2n}=\frac{(2n)!}{n!(n+1)!}\cdot q^{n+1}+o_{n}(q^{n+1}).
		$$
	\end{proposition}

	\begin{proof}
		We use induction on $n$ to prove the identity. Note that for $q\equiv1\pmod{3},$ we have
		$$
		12\sum\limits_{s\equiv q+1\pmod 9}H^\ast\left(\frac{4q-s^2}{9}\right)s^{2n}=6\sum\limits_{s\in\Z}H^\ast\left(\frac{4q-s^2}{9}\right)s^{2n}.
		$$
		For $n=0,$ this follows immediately from Lemma~\ref{cardinality-1} and Theorem~\ref{Schoof}. Now, 
		assume that the identity is true for $n'<n.$ Let $f(\tau):={\mathcal{H}}(9\tau) $ and $g(\tau):=\sum\limits_{n\in\Z}q_{\tau}^{n^2}.$ Then $f$ is a harmonic Maass form of weight $3/2$ on the congruence subgroup $\Gamma_0(36)$ and $g$ is the usual weight $1/2$ theta function on the congruence subgroup $\Gamma_0(4).$ Then, for any $\nu\geq 1,$ we consider $\pi_{\text{hol}}([f,g]_{\nu}).$ Using \eqref{HolomorphicProjectionDecomposition}, Lemma  V.2.6 of \cite{mertens}, Lemma \ref{HolProjExplicit1}, Theorem \ref{HolProjExplicit2} and Proposition V.2.7 of \cite{mertens},
		we have that the $m$-th Fourier coefficient of $\pi_{\text{hol}}([f,g]_\nu)$ is
		
		\begin{align}\label{Fourier-1}
			&\sum\limits_{\mu=0}^{\nu}(-1)^{\nu-\mu}{{\nu+1/2}\choose{\mu}}{{\nu-1/2}\choose{\nu-\mu}}\sum\limits_{s\in \Z} s^{2\mu}(m-s^2)^{\nu-\mu}H^\ast\left(\frac{m-s^2}{9}\right)\notag\\
			&+C_1(\nu)\cdot\delta(m) m^{\nu+1/2}+C_2(\nu)\sum\limits_{\substack{t^2-9l^2=m \\ t,l\geq 1}}(t-3l)^{2\nu+1},
		\end{align}
		where 
		$$
		\delta(m):=\begin{cases}
			1, & \text{if}\ m \ \text{is \ a \  square}\\
			0, & \ \text{otherwise,}
		\end{cases}
		$$

		\noindent and where $C_1(\nu)=\frac{1}{3}\cdot 2^{-2\nu-1}\cdot{{2\nu}\choose{\nu}}$ and $C_2(\nu)=\frac{1}{3}\cdot 2^{-2\nu}\cdot{{2\nu}\choose{\nu}}.$ 
		It is easy to see that the sum $\sum\limits_{\substack{t^2-9l^2=m \\ t,l\geq 1}}(t-3l)^{1+2\nu}$ runs over all the divisors $d$ of $m$ such that $d\leq\sqrt{m}.$ Therefore, we have
		$$
		\sum\limits_{\substack{t^2-9l^2=m \\ t,l\geq 1}}(t-3l)^{1+2\nu}\leq m^{1/2+\nu}\sigma_0(m),
		$$ 
		where $\sigma_0(m)$ denotes the number of divisors of $m.$ Moreover, we have $\sigma_0(m)=o(m^{\epsilon})$ for every $\epsilon>0$ (for example, see \cite[pg 296]{Apostal}). This implies that $$\sum\limits_{\substack{t^2-9l^2=m \\ t,l\geq 1}}(t-3l)^{1+2\nu}=o_\nu(m^{\nu+1}).$$ Therefore, taking $m=4q$, where $q\equiv1\pmod{3}$ and using Theorem~\ref{Deligne}, we have that
		
		\begin{align}\label{Fourier-3}
			\sum\limits_{\mu=0}^{\nu}(-1)^{\nu-\mu}{{\nu+1/2}\choose{\mu}}{{\nu-1/2}\choose{\nu-\mu}}\sum\limits_{s\in \Z} s^{2\mu}(4q-s^2)^{\nu-\mu}H^\ast\left(\frac{4q-s^2}{9}\right)=o_{\nu}(q^{\nu+1}).
		\end{align}
		Merten's proved (see pg 60 of \cite{mertens}) that for $0\leq\mu\leq \nu,$
		
		$${{\nu+1/2}\choose{\nu-\mu}}{\nu-1/2\choose\mu}=4^{-\nu}{2\nu\choose\nu}
		{2\nu+1\choose2\mu+1}.$$ Now, replacing $\mu$ by $\nu-\mu$ in the above identity we write
		
		$${{\nu+1/2}\choose{\mu}}{\nu-1/2\choose\nu-\mu}=4^{-\nu}{2\nu\choose\nu}
		{2\nu+1\choose2\nu-2\mu+1}.$$ Using this identity in \eqref{Fourier-3}, we have

		\begin{align}
			& 4^{-\nu}{2\nu\choose\nu}\sum\limits_{\mu=0}^{\nu}
			{2\nu+1\choose2\nu-2\mu+1}\sum\limits_{s\in\Z} H^\ast\left(\frac{4q-s^2}{9}\right)
			\sum\limits_{k=0}^{\nu-\mu}{\nu-\mu\choose k}(-1)^{k}s^{2(\nu-k)}(4q)^{k}= o_{\nu}(q^{\nu+1}).\notag
		\end{align}
		By using Lemma \ref{CombLemma3Equation} and letting $\nu=n$ we deduce that
		
		\begin{align}\label{Fourier-4}
			\sum\limits_{k=1}^{n}(-1)^k \frac{(2n-k)!q^k}{k!(2n-2k)!}\sum\limits_{s\in\Z} H^{\ast}  
			\left(\frac{4q-s^2}{9}\right)s^{2(n-k)}+\sum\limits_{s\in\Z} H^{\ast}  
			\left(\frac{4q-s^2}{9}\right)s^{2n}=o_{n}(q^{n+1}).
		\end{align}
		Cohen proved (see p. 284 of \cite{cohen}) that 
		$$
		\sum\limits_{k=0}^{n}(-1)^k\frac{(2n-k)!}{k!(n-k)!(n+1-k)!}=0.
		$$
		By making use of the induction hypothesis and the above identity in \eqref{Fourier-4}, the claim follows.

	\end{proof}


	The following proposition provides an asymptotic formula for the odd moments that appear in Theorem \ref{2F1-moments}.
	
	\begin{proposition}\label{AsymptoticsOdd}
		If $m$ is a positive odd integer, $p$ is a prime and $r\geq1$ is fixed, then as $p\to\infty$ with $q:=p^r\equiv1\pmod{3},$ we have
		$$
		12\sum\limits_{s\equiv q+1\pmod 9}H^\ast\left(\frac{4q-s^2}{9}\right)s^m=o_{m,r}(q^{m/2+1}).
		$$
	\end{proposition}

	\begin{proof}[Proof of Proposition \ref{AsymptoticsOdd}]
		Here we prove the case $q\equiv1\pmod 9$ and the other cases are proved in analogous fashion. Let $f(\tau)=\mathcal{H}(9\tau)$ and $g(\tau)=\vartheta(\tau)=\sum\limits_{n\in\Z} n\leg{n}{3}q_{\tau}^{n^2}.$ Then $f$ is a harmonic Maass form of weight $3/2$ on $\Gamma_0(36)$ and $g$ is a modular form of weight $3/2$ on $\Gamma_0(36)$ with nebentypus $\leg{\cdot}{3}.$ By Proposition \ref{HolProjProp} and Proposition \ref{BracketProposition} we have that $\pi_{\text{hol}}([f, g]_{\nu})$ is a cusp form of weight $2\nu+3$ on $\Gamma_0(36)$ with nebentypus $\leg{\cdot}{3}.$ Now by \eqref{HolomorphicProjectionDecomposition}, Lemma \ref{HolProjExplicit1}, Lemma \ref{HolProjExplicit2} and Lemma \ref{lemma-odd-1}, its Fourier expansion is given by
		
		\begin{align}\label{expansion}
			&\sum\limits_{n=0}^{\infty}
			\left(\sum\limits_{\mu=0}^{\nu}(-1)^{\mu}{\nu+\frac{1}{2}\choose\mu}
			{\nu+\frac{1}{2}\choose\nu-\mu}\sum\limits_{s\in\Z}\leg{s}{3}s^{2\nu-2\mu+1}(n-s^2)^{\mu}H^{*}
			\left(\frac{n-s^2}{9}\right)\right)q_{\tau}^{n}\notag\\
			&+C_3(\nu)\sum\limits_{n=1}^{\infty} n^{2\nu+2}\leg{n}{3}q_{\tau}^{n^2}+C_4(\nu)\sum\limits_{n=1}^\infty\left(\sum\limits_{\substack{t^2-9\ell^2=n \\ t,\ell\geq 1}}\leg{t}{3}(t-3\ell)^{2\nu+2}\right)q_{\tau}^n,
		\end{align}
		where $C_3(\nu)=\frac{1}{3}\cdot 2^{-2\nu-3}\cdot{{2\nu+2}\choose{\nu+1}}$ and $C_4(\nu)=\frac{1}{3}2^{-2\nu-1}{{2\nu+1}\choose{\nu+1}}.$ We now use induction on $\nu$.

		If $\nu=0$ then $\pi_{\text{hol}}(f\cdot g)$ is a cusp form of weight 3 on $\Gamma_0(36)$  with nebentypus $\leg{\cdot}{3}.$ Using \eqref{expansion} for the Fourier expansion of $\pi_{\text{hol}}(f\cdot g)$ and Deligne's theorem we have
		
		\begin{align}
			&\sum_{n=1}^{\infty}\left(-\sum\limits_{s\equiv q+1\pmod{9}}H^*\left(\frac{n-s^2}{9}\right)s\right)q_{\tau}^n
			+C_3(0)\sum_{n=1}^{\infty}\left(\delta(n)\cdot n\leg{\sqrt{n}}{3}\right)q_{\tau}^n
			\notag\\
			&+C_4(0)\sum_{n=1}^{\infty}
			\left(\sum\limits_{\substack{t^2-9\ell^2=n \\ t,\ell\geq 1}}\leg{t}{3}(t-3\ell)^{2}\right)q_{\tau}^n=o(n^{3/2}),\notag
		\end{align}
		where $$
		\delta(n)=\begin{cases}
			1, & \text{if} \ n=\square\\
			0, & \text{otherwise}.
		\end{cases}
		$$
		
		\noindent Note that 
		$\sum\limits_{\substack{t^2-9\ell^2=n \\ t,\ell\geq 1}}(t-3\ell)^{2}\leq n\cdot\sigma_0(n)$, where 
		$\sigma_0(n)$ is the divisor function and therefore
		$$\sum\limits_{\substack{t^2-9\ell^2=n \\ t,\ell\geq 1}}(t-3\ell)^{2}=o(n^{3/2}).$$ 
		Therefore, the claim holds for $m=1.$ Now, suppose that the result is true for $m'<m.$ Then we conclude the result for $m$ by using the expression \eqref{expansion} for $\nu=(m-1)/2$, Theorem~\ref{Deligne} and the induction hypothesis.

	\end{proof}
	
	\subsection{Asymptotic Moment formulas of $a_{q}^{\hes}(\lambda)$ when $q\equiv2\pmod{3}$}
	
	\begin{proposition}\label{AsymptoticEven-2}
		If $n\geq 0$ and  $r\geq 1$ is fixed, then as $p\to\infty$ with $q:=p^r\equiv 2\pmod 3,$ we have
		\begin{align}
			\sum\limits_{s\equiv q+1\pmod3}H^*(4q-s^2)s^{2n}=\frac{(2n!)}{n!(n+1)!}q^{n+1}+o_n(q^{n+1}).
		\end{align}
	\end{proposition}
	\begin{proof}
		This follows from Theorem 1.1 of \cite{KP} and Lemma 4.1 (1) of \cite{BKP2}.
		
	\end{proof}

	\begin{proposition}\label{AsymptoticOdd-2}
		If $m\geq 1$ is odd and $r\geq 1$ is fixed, then as $p\to\infty$ with $q:=p^r\equiv 2\pmod 3,$ we have 
		\begin{align}
			\sum\limits_{s\equiv q+1\pmod3}H^*(4q-s^2)s^m=o_m(q^{m/2+1}).
		\end{align}
	\end{proposition}
	\begin{proof}
		This follows from Theorem 1.3 (2) of \cite{BKP3}.
	\end{proof}

	\subsection{Proof of the main results}
	\begin{proof}[Proof of Theorem \ref{2F1-moments}]
		By using Proposition \ref{2F1-moment-class-number} we can write the power moments of the values of ${_2F_1(\lambda)_q}$ functions as weighted sums of class numbers. Then applying Proposition \ref{AsymptoticsEven}, Proposition \ref{AsymptoticsOdd}, Lemma \ref{Bound-1} and the classical Hasse--Weil bound $ a_{p^r}^{\hes}(\lambda)\in[-2\sqrt{p^r}, 2\sqrt{p^r}]$ we obtain the required formula.
		
	\end{proof}
	
	\begin{proof}[Proof of Corollary \ref{Distribution2F1}]
		This follows from Theorem \ref{2F1-moments} and Lemma 6.1 (1) of \cite{KHN}.
	\end{proof}
	
	\begin{proof}[Proof of Theorem \ref{MomentsOfTrace1}]
		If $q\equiv1\pmod{3},$ then the claim follows from Theorem \ref{Bridge} and Theorem \ref{2F1-moments}. If $q\equiv2\pmod{3},$ then the claim follows from Proposition \ref{Trace-moment-class-number} (3), Proposition \ref{AsymptoticEven-2}, Proposition \ref{AsymptoticOdd-2} and Lemma \ref{Bound-1}.
	\end{proof}

	\begin{proof}[Proof of Corollary \ref{TraceDistribution}]
		This follows from Theorem \ref{MomentsOfTrace1} and Lemma 6.1 (1) of \cite{KHN}.
		
	\end{proof}

	\section{Acknowledgements}
	\noindent  We thank Jerome Hoffman for motivating us to take up the above problem. The first author thanks the Thomas Jefferson Fund and grants from the NSF
	(DMS-2002265 and DMS-2055118). 
	The second author is supported by Science and Engineering Research Board [SRG/2023/000930]. The fourth author is supported by Science and Engineering Research Board [CRG/2023/003037].


\begin{thebibliography}{99}		
		
		
		\bibitem{ahlgren-ono}
		S. Ahlgren and K. Ono, {\it Modularity of a certain Calabi-Yau threefold}, Montash. Math. {\bf 129} (2000), no. 3,  177--190.
		
		\bibitem{AO} S. Ahlgren and K. Ono, \emph{A Gaussian hypergeometric series evaluation and Ap\'ery number congruences},
		J. Reine Angew. Math. \textbf{518} (2000), 187-212.
		
		\bibitem{ono-3}
		S. Ahlgren, K. Ono, and D. Penniston, {\it Zeta functions of an infinite family of $K3$ surfaces}, Amer. J. Math.,
		{\bf  124} (2) (2002), 353--368.
		
		\bibitem{Askey}
		G. E. Andrews, R. Askey and R. Roy, \emph{Special functions}, Encyclopedia of mathematics and its applications, Cambridge Univ. Press, 1999.
		
		\bibitem{Apostal}
		T. Apostol, {\it Introduction to Analytic Number Theory}, Springer-Verlag, 1976.
		
		
		
		\bibitem{bailey}
		W. Bailey, {\it Generalized hypergeometric series}, Cambridge Tracts in Mathematics and Mathematical Physics 
		\textbf{32}, Cambridge Univ. Press, Cambridge, 1935.
		
		\bibitem{B-K} R. Barman and G. Kalita, \emph{Elliptic curves and special values of Gaussian hypergeometric series}, J. Number Th., \textbf{133} (2013), 3099-3111.
		
		\bibitem{BCM}
		F. Beukers, H. Cohen, and A. Mellit, {\it Finite hypergeometric functions, Pure Appl. Math. Q. \textbf{11} (2015), no. 4, 559--589.
			
			
			
		}
		
		
		\bibitem{birch}
		B. J. Birch, {\it How the number of points of an elliptic curve over a fixed prime field varies}, J. London Math. Soc., {\bf 43} (1968), 57--60.
		
		\bibitem{BFOR} K. Bringmann, A. Folsom, K. Ono, and L. Rolen, \emph{Harmonic Maass forms and mock modular forms: Theory and applications}, Amer. Math. Soc. Colloq. \textbf{64}, Amer. Math. Soc., Providence, 2017.
		
		\bibitem{BKP2}
		K. Bringmann, B. Kane, and S. Pujahari, \emph{Formulas for moments of class numbers in arithmetic progressions}, Acta Arithmetica, \textbf{207}, (2023), no. 1, 19--38.
		
		\bibitem{BKP3}
		K. Bringmann, B. Kane, and S. Pujahari, {\it Odd moments for the trace of Frobenius and the Sato--Tate conjecture in arithmetic progressions}, Finite Fields Appl. \textbf{98} (2024), 1--26.	
		
		\bibitem{SatoTateProof}
		L. Clozel, M. Harris, N. Shepherd-Barron, and R. Taylor (2008). {\it Automorphy for some $\ell$-adic lifts of automorphic mod $\ell$ Galois representations}, Publ. Math. Inst. Hautes Études Sci. {\bf 108} (2008), 1--181.
		
		\bibitem{cohen}
		H. Cohen, {\it Sums involving the values at negative integers of $L$-functions of quadratic characters}, Math. Ann., {\bf 217} (1975), 271--285.
		
		\bibitem{stromberg}
		H. Cohen and F. Str\"omberg, \emph{ Modular forms: A classical approach}, Graduate Studies in Mathematics, Vol \textbf{179}, Amer. Math. Soc., Providence, 2017.
		
		\bibitem{F-R}
		R. Farashahi, {\it On the number of distinct Legendre, Jacobi, Hessian and Edwards curves} (extended abstract), in: Proceedings of the Workshop on Coding Theory and Cryptology (WCC 2011), (2011), 37--46.
		
		\bibitem{fop}
		S. Frechette, K. Ono, and M. Papanikolas, {\it Gaussian hypergeometric functions and traces of Hecke operators}, Int. Math. Res. Not. {\bf (60)} (2004), 3233--3262.
		
		\bibitem{Fuselier} J. Fuselier, \emph{Hypergeometric functions over $\F_p$ and relations to elliptic curves and modular forms},
		Proc. Amer. Math. Soc. \textbf{138} (2010), 109-123.
		
		\bibitem{FuselierHecke} J. Fuselier, \emph{Traces of Hecke operators in level 1 and Gaussian hypergeometric functions},
		Proc. Amer. Math. Soc. \textbf{141} (2013), 1871-1881.
		
		\bibitem{Long}
		J. Fuselier, L. Long, R. Ramakrishna, H. Swisher and F.-T. Tu, \emph{Hypergeometric functions over finite fields},  Mem. Amer. Math. Soc. \textbf{280} (2022), no. 1382.
		
		
		\bibitem{Gauss}
		C. F. Gauss, {\it Disquisitiones generales circa seriem infinitam $1+\frac{\alpha\beta}{1\cdot\gamma}x
			+\frac{\alpha(\alpha+1)\beta(\beta+1)}{1\cdot2\cdot\gamma(\gamma+1)}x^2
			+\frac{\alpha(\alpha+1)(\alpha+2)\beta(\beta+1)(\beta+2)}{1\cdot2\cdot3\cdot\gamma(\gamma+1)(\gamma+2)}x^3+$ etc.}, Cambridge Univ. Press, Cambridge, 2011, 233-279.
		
		\bibitem{GreenePhD} J. Greene, \emph{Character sum analogues for hypergeometric and generalized hypergeometric functions over finite fields}, Thesis (Ph.D.)-University of Minnesota, 1984.
		
		\bibitem{greene}
		J. Greene, {\it Hypergeometric functions over finite fields}, Trans. Amer. Math. Soc. {\bf 301} (1987), no. 1, 77--101.
		
		\bibitem{griffin-ono-tsai}
		M. Griffin, K. Ono, and W. Tsai {\it Heights of points on elliptic curves over $\Q$}, Proc. Amer. Math. Soc. \textbf{149} (2021), no. 12, 5093–5100.
		
		\bibitem{rosen}
		K. Ireland and M.Rosen, {\it A Classical introduction to modern number theory}, 2nd ed,  Springer,  GTM Vol. 84, New York, (1990).
		
		
		\bibitem{KP} B. Kane, and S. Pujahari, \emph{Distribution of moments of Hurwitz class numbers in arithmetic progressions and holomorphic projection,} Trans. Amer. Math. Soc. \textbf{376} (2023), no. 8, 5503--5519.
		
		\bibitem{Katz}
		N. Katz, {\it Exponential sums and differential equations}, Annals of Math. Studies \textbf{124}, Princeton Univ. Press, Princeton, 1990.
		
		\bibitem{Katz2}
		N. Katz, {\it Rigid local systems}, Annals of Math. Studies \textbf{139}, Princeton Univ. Press, Princeton, 1996.
		
		\bibitem{Lennon1} C. Lennon, \emph{Trace formulas for Hecke operators, Gaussian hypergeometric functions, and the modularity of a threefold}, J. Numb. Th. \textbf{131} (2011), 2320-2351.
		
		\bibitem{Lennon} C. Lennon, \emph{Gaussian hypergeometric evaluations of traces of Frobernius for elliptic curves},
		Proc. Amer. Math. Soc. \textbf{139} (2011), 1931-1938.
		
		
		
		\bibitem{McCarthy} D. McCarthy, \emph{$_3F_2$ hypergeometric series and periods of elliptic curves}, Int. J. of Number Th., \textbf{6} (2010), no. 3, pages 461-470.
		
		
		
		\bibitem{McCarthyPJM} D. McCarthy, \emph{The trace of Frobenius of elliptic curves and the $p$-adic Gamma-function}, Pacific J. Math. \textbf{261} (2013), 219-236.
		
		\bibitem{McCarthyRIMS}	D. McCarthy, \emph{The number of $\F_q$-points on Dwork hypersurfaces and hypergeometric functions}, Res. in Math. Sci., \textbf{4} (2017), no. 4, 1-15.
		
		
		
		\bibitem{MOS} D. McCarthy, R. Osburn, and A. Straub, \emph{Sequences, modular forms, and cellular integrals}, Math. Proc. Cambridge Philos. Soc. \textbf{168} (2020), 379-404.
		
		\bibitem{MP} D. McCarthy and M. Papanikolas, \emph{A finite field hypergeometric function associated to eigenvalues of a Siegel eigenform}, Int. J. Number Th. \textbf{11} (2015), 2431-2450.
		
		\bibitem{mertens}
		M. Mertens, {\it Mock modular forms and class numbers of quadratic forms}, 
		Thesis (Ph.D.)-University of Cologne, 2014, 1--86.
		
		
		\bibitem{M-W} 
		D. Moody and H. Wu, {\it Families of elliptic curves with rational 3-torsion}, J. Math. Cryptol. 5 (2011), 225--246.
		
		
		\bibitem{ono}
		K. Ono, {\it Values of Gaussian hypergeometric series}, Trans. Amer. Math. Soc. {\bf 350} (3) (1998), 1205--1223.
		
		\bibitem{ono-book}
		K. Ono, {\it The web of modularity: Arithmetic of the coefficients of modular forms and $q$-series}, CBMS, Regional Conference series in Mathematics, 102, Amer. Math. Soc., Providence, 2004.
		
		\bibitem{KHN}
		K. Ono, H. Saad and N. Saikia, \textit{Distribution of values of Gaussian hypergeometric functions}, Pure Appl. Math. Q. 19 (2023), no. 1, 371--407.
		
		
		\bibitem{PujahariSaikia} S. Pujahari and N. Saikia, \emph{Traces of Hecke operators in level 1 and $p$-adic hypergeometric functions}, Ramanujan J. \textbf{52} (2020), 519-539.
		
		
		\bibitem{Rouse} J. Rouse, \emph{Hypergeometric functions and elliptic curves}, Ramanujan J., \textbf{12} (2006), no. 2, pages 197-205.
		
		\bibitem{schoof}
		R. Schoof, {\it Nonsingular plane cubic curves over finite fields}, J. Comb. Theory Ser. A {\bf 46} (2) (1987), 183--211. 
		
		\bibitem{TY} F.-T. Tu and Y. Yang, \emph{Evaluation of certain hypergeometric functions over finite fields}, SIGMA \textbf{14} (2018), Art.  50.
		
		\bibitem{Zagier1} D. Zagier, \emph{Nombres de classes et formes modulaires de poids 3/2}, C. R. Acad. Sci. Paris S\'er. A-B \textbf{281} (1975), Ai, A883-A886.
		
		
		\bibitem{Z} W. Zudilin, \emph{A hypergeometric version of the modularity of rigid Calabi-Yau manifolds},
		SIGMA \textbf{14} (2018), Art. 86.
		
		
		
	\end{thebibliography}
\end{document}